\documentclass{article}

\usepackage{arxiv}

\usepackage[utf8]{inputenc} 
\usepackage[T1]{fontenc}    
\usepackage{hyperref}       
\usepackage{url}            
\usepackage{booktabs}       
\usepackage{amsfonts}       
\usepackage{nicefrac}       
\usepackage{microtype}      
\usepackage{lipsum}
\usepackage{amsmath, amssymb, amsthm, amsfonts, fancyhdr, graphicx, subcaption, xcolor, mdframed, array, tabu, multirow}
\usepackage{color}
\usepackage[export]{adjustbox}
\DeclareMathOperator{\Ker}{Ker}
\DeclareMathOperator{\Rg}{Rg}
\DeclareMathOperator{\rank}{rank}
\DeclareMathOperator{\id}{id}
\renewcommand{\null}{\text{null}}
\newtheorem{theorem}{Theorem}

\theoremstyle{definition}
\newtheorem{example}{Example}

\title{An Algorithmic Approach to Solving B = BCX + YAB Using Quotient Spaces}

\author{
  Alex Taylor\\
  Department of Mathematical Sciences\\
  University of Texas at Dallas,\\ 800 W Campbell Rd, Richardson, TX, 75080\\
  \texttt{Alex.Taylor@utdallas.edu}
}

\begin{document}
\large
\maketitle

\begin{abstract}
One well-known necessary and sufficient condition for equality in the Frobenius rank inequality due to Tian and Styan [1] is that the matrix equation $B = BCX + YAB$ be solvable for $X$ and $Y$. We develop an algorithm to construct the matrices $X$ and $Y$ using a quotient space formulation of the Frobenius rank inequality, and provide several necessary and sufficient conditions for solvability. \footnote{Some of the work on the present paper was carried out during the IMPACT summer camp at The University of Texas at Dallas. We gratefully acknowledge funding and support from the NSF.}
\end{abstract}

\section{Introduction}
Denote by $\mathbb{M}_{m,n}$ the space of $m \times n$ matrices over a field $\mathbb{F}$. The Frobenius rank inequality states that for $A \in \mathbb{M}_{m,n}, B \in \mathbb{M}_{n,p},$ and $ C \in \mathbb{M}_{p,q}$,
\begin{equation}
\rank(ABC) + \rank(B) \geq \rank(AB) + \rank(BC).
\end{equation}
A known proof of this fact utilizes quotient spaces. Consider the linear map
\[ T: \Rg(B)/\Rg(BC) \to \Rg(AB)/\Rg(ABC), \]
defined by $T([x]) = [Ax]$. Since $T$ is surjective, we have
\[ \dim(\Rg(B)/\Rg(BC)) \geq \dim(\Rg(AB)/\Rg(ABC), \]
hence
\[ \rank(B) - \rank(BC) \geq \rank(AB) - \rank(ABC), \]
yielding the inequality in (1). It was proven by Tian and Styan in [1] that equality in (1) occurs if and only if there exist matrices $X$ and $Y$ such that
\begin{equation}
B = BCX + YAB.
\end{equation}
It is clear that equality also occurs precisely when $T$ is injective (i.e., an isomorphism). In the following section we establish the equivalence of this fact to several conditions and construct matrices $X$ and $Y$ so that (2) holds.

\section{Equality In The Frobenius Rank Inequality}
\label{sec:headings}

For reference we list some basic facts and notations regarding quotient spaces. In this paper we denote by $V/W = \{ [v] : v \in V \}$ the vector space of equivalence classes induced by a subspace $W \subset V$ with the natural operations $[v] + [w] = [v + w]$ and $c[v] = [cv]$ for all $v,w \in V$ and $c \in \mathbb{F}$.

\begin{theorem}
Let $W$ be a subspace of $V$. If $\{w_{1}, \ldots, w_{m}\}$ is a basis for $W$ and $\{w_{1}, \ldots, w_{m}, v_{m+1}, \ldots, v_{n}\}$ is an extension to a basis for $V$, then $\{[v_{m+1}],\ldots, [v_{n}]\}$ is a basis for $V/W$. In particular, we have
\[ \dim(V/W) = \dim V - \dim W. \]
\end{theorem}

Let $A, B, $ and $C$ be matrices as in (1), and let $L_{A}: \Rg(B) \to \Rg(AB)$ be the linear transformation defined by $L_{A}(x) = Ax$. Clearly the aforementioned linear map $T$ is a quotient map induced by $L_{A}$. Let $\mathcal{B}_{1}$ and $\mathcal{B}_{2}$ be bases for $\Rg(B)$ and $\Rg(AB)$ of size $n_{1}$ and $n_{2}$ respectively, that extend bases for $\Rg(BC)$ and $\Rg(ABC)$ of size $m_{1}$ and $m_{2}$, respectively. Then the lower right $(n_{1}-m_{1})\times (n_{2}-m_{2})$ block of the matrix representation of $L_{A}$ with respect to $\mathcal{B}_{1}$ and $\mathcal{B}_{2}$ is precisely the matrix representation of $T$ with respect to the bases for $\Rg(B)/\Rg(BC)$ and $\Rg(AB)/\Rg(ABC)$ given by Theorem 1. This observation gives us the following necessary and sufficient condition for equality (2).

\begin{theorem}
Let $\rank(B) = n_{1}$, $\rank(AB) = n_{2}$, $\rank(BC) = m_{1}$ and $\rank(ABC) = m_{2}$. Then equality holds in (1) if and only if the lower right $(n_{1}-m_{1})\times (n_{2}-m_{2})$ block of the matrix representation for $L_{A}$ is invertible.
\end{theorem}

\begin{theorem}
Let $A,B$ be matrices of appropriate sizes. Then
\[
\rank(AB) = \rank(B) - \dim(\Rg(B)\cap \Ker(A)).
\]
\end{theorem}
\begin{proof}
Consider the linear map
\[
\varphi: \Rg(B)/ \Rg(B)\cap\Ker(A) \to \Rg(AB)
\]
defined by $\varphi([Bx]) = ABx$. Clearly $\varphi$ is an isomorphism, so we have
\begin{eqnarray*} 
\rank(AB) &=& \dim(\Rg(B)/ \Rg(B)\cap\Ker(A)) \\
&=& \rank(B) - \dim(\Rg(B)\cap\Ker(A))
\end{eqnarray*}
as desired.
\end{proof}

The following necessary and sufficient condition for equality in the Frobenius rank inequality (1) follows directly from Theorem 3.
\begin{theorem}
\[ \rank(ABC) + \rank(B) = \rank(AB) + \rank(BC) \]
if and only if
\[ \Rg(B)\cap \Ker(A) = \Rg(BC)\cap \Ker(A). \]
\end{theorem}
\begin{proof}
Notice that, since $\Rg(BC)\cap \Ker(A) \subseteq \Rg(B)\cap \Ker(A)$, we have
\[
\Rg(BC)\cap \Ker(A) = \Rg(B)\cap \Ker(A)
\]
if and only if
\[
\dim(\Rg(BC)\cap \Ker(A)) = \dim(\Rg(B)\cap \Ker(A)),
\]
if and only if, by Theorem 1,
\[
\rank(BC) - \rank(ABC) = \rank(B) - \rank(AB),
\]
as desired.
\end{proof}

We will use the following notation in the next two results. Given a matrix $B$ with rank $r$, define $D_{B}$ to be a matrix whose columns are some $r$ linearly independent columns of $B$. Now if $\{v_{1},\ldots,v_{s}\}$ is a basis for $\Ker(AD_{B})$, then $\mathcal{B} = \{D_{B}v_{1}, D_{B}v_{2}, \ldots, D_{B}v_{s}\}$ is a basis for $\Rg(B)\cap \Ker(A)$. To prove this, it suffices to show that $\mathcal{B}$ is linearly independent and that $\dim (\Rg(B) \cap \Ker(A)) = s.$ If $V_{B}$ is the matrix whose columns are $\{v_{1}, \ldots, v_{s} \}$ then the vectors of $\mathcal{B}$ are precisely the columns of $D_{B}V_{B}$, which has full column rank because $D_{B}$ and $V_{B}$ have full column rank. Now  $\Rg(B) = \Rg(D_{B})$, so we have
\begin{eqnarray*}
\dim(\Rg(B)\cap \Ker(A)) &=& \dim(\Rg(D_{B})\cap \Ker(A)) \\
&=& \rank(D_{B}) - \rank(AD_{B}) \\
&=& r - \null(D_{B}) - (r - \null(AD_{B})) \\
&=& \null(AD_{B}) - \null(D_{B}) \\
&=& \null(AD_{B}) \\
&=& s,
\end{eqnarray*}
so $\mathcal{B}$ is a basis for $\Rg(B)\cap\Ker(A)$.

\begin{theorem}
Let $V_{B}$ and $V_{BC}$ denote matrices whose columns constitute bases for $\Ker(AD_{B})$ and $\Ker(AD_{BC})$, respectively. Then equality in (1) holds if and only if
\[
D_{B}V_{B} = D_{BC}V_{BC}Z
\]
for some matrix $Z$.
\end{theorem}
\begin{proof}
Note that $\Rg(D_{B}V_{B}) = \Rg(B)\cap\Ker(A)$ and $\Rg(D_{BC}V_{BC}) = \Rg(BC)\cap\Ker(A)$ by the above remarks. Also, it is a general fact that for any two matrices $M$ and $N$, $\Rg(M) \subseteq \Rg(N)$ if and only if $M = NZ$ for some matrix $Z$.  Thus the result follows from Theorem 4.
\end{proof}
We now prove that equality holds in the Frobenius rank inequality (1) precisely when equation (2) holds using Theorem 3, Theorem 4, and Theorem 5. This is the main purpose of the present note.

\begin{theorem}
\[ \rank(ABC) + \rank(B) = \rank(AB) + \rank(BC) \]
if and only if there exist matrices $X$ and $Y$ such that
\[ B = BCX + YAB. \]
\end{theorem}
\vspace{-0.5em}
\begin{proof}
First suppose that $B = BCX + YAB$ for some matrices $X$ and $Y$. There exists a matrix $E$ such that $BE = D_{B}V_{B}$ since $\Rg(D_{B}V_{B}) = \Rg(B)\cap\Ker(A)$. Thus $ABE = AD_{B}V_{B} = 0_{m\times s}$. Hence
\begin{eqnarray*}
D_{B}V_{B} &=& BE \\
&=& BCXE + YABE \\
&=& BCXE.
\end{eqnarray*}
Also $ABCXE = ABE = 0_{m\times s}$, so $\Rg(BCXE) \subset \Rg(BC)\cap\Ker(A)$ and we have
\begin{eqnarray*}
D_{B}V_{B} &=& BCXE \\
&=& D_{BC}V_{BC}Z,
\end{eqnarray*}
so equality holds in (1) by Theorem 5.

Conversely, suppose that $\Rg(B)\cap\Ker(A)=\Rg(BC)\cap\Ker(A)$. As in Theorem 3, define $\varphi:\Rg(B)/\Rg(B)\cap\Ker(A)\to\Rg(AB)$ by $\varphi([Bx])=ABx$. Extend $\mathcal{B}$ to a basis $\{D_{B}v_{1}, \ldots, D_{B}v_{s}, D_{B}v_{s+1}, \ldots, D_{B}v_{r} \}$ for $\Rg(B)$. Then $\{[D_{B}v_{s+1}], \ldots, [D_{B}v_{r}]\}$ is a basis for the quotient space $\Rg(B)/\Rg(B)\cap\Ker(A)$ and $\varphi$ produces a basis $\{AD_{B}v_{s+1}, \ldots, AD_{B}v_{r}\}$ for $\Rg(AB)$. Define a linear map $T_{Y}: \Rg(AB) \to \Rg(B)$ by $T(AD_{B}v_{k}) = D_{B}v_{k}$ for each $k \in \{s+1, \ldots, r\}$, and extend linearly to all of $\Rg(AB)$. Let $Y$ be the matrix of $T_{Y}$ with respect to these bases for $\Rg(AB)$ and $\Rg(B)$.

We want to choose $X$ such that $Bz = BCXz + YABz$ for every $z \in \mathbb{F}^{p}$. Write
\[ Bz = \sum_{j = 1}^{s}c_{j}D_{B}v_{j} + \sum_{k=s+1}^{r} d_{k}D_{B}v_{k} \]
and note that
\begin{eqnarray*}
YABz &=& YA\left(\sum_{j = 1}^{s}c_{j}D_{B}v_{j} + \sum_{k=s+1}^{r} d_{k}D_{B}v_{k}\right) \\
&=& Y\left(\sum_{j = 1}^{s}c_{j}AD_{B}v_{j} + \sum_{k=s+1}^{r} d_{k}AD_{B}v_{k}\right) \\
&=& Y\left(0 + \sum_{k=s+1}^{r} d_{k}AD_{B}v_{k}\right) \\
&=& \sum_{k=s+1}^{r} d_{k}D_{B}v_{k},
\end{eqnarray*}
since each $v_{j} \in \Ker(AD_{B})$ for $j \in \{1, \ldots, s \}$. Hence it is clear that $Bz = BCXz + YABz$ if and only if $BCXz = \sum_{j = 1}^{s}c_{j}D_{B}v_{j}$.

Now $c_{j}D_{B}v_{j} \in \Rg(BC)\cap\Ker(A)$ for each $j$ because $c_{j}D_{B}v_{j} \in \Rg(B)\cap\Ker(A)$ for each $j$, so for each $j$ there exists $\tilde{v}_{j} \in \mathbb{F}^{q}$ such that $D_{B}v_{j} = BC\tilde{v}_{j}$, and we can define the linear map $T_{M}: \Rg(B) \to \mathbb{F}^{q}$ by
\begin{align*} 
T_{M}(D_{B}v_{j}) &= \tilde{v}_{j} \hspace{3mm} \mbox{ for each $j \in \{1, \ldots, s\}$, and } \\ 
T_{M}(D_{B}v_{k}) &= 0 \hspace{3mm} \mbox{ for each $k \in \{s+1, \ldots, r\}$},
\end{align*}
and after extending $T_{M}$ linearly to $Rg(B)$ we define $M \in \mathbb{M}_{q,n}$ to be the matrix representation of $T_{M}$. Let $X = MB$, then
\begin{eqnarray*}
BCXz &=& BCMBz \\
&=& BCM\left(\sum_{j = 1}^{s}c_{j}D_{B}v_{j} + \sum_{k=s+1}^{r} d_{k}D_{B}v_{k}\right) \\
&=& BC\left(\sum_{j = 1}^{s}c_{j}MD_{B}v_{j} + \sum_{k=s+1}^{r} d_{k}MD_{B}v_{k}\right) \\
&=& BC\left(\sum_{j=1}^{s}c_{j}MD_{B}v_{j} + 0 \right) \\
&=& \sum_{j=1}^{s}c_{j}BC\tilde{v}_{j} \\
&=& \sum_{j=1}^{s}c_{j}D_{B}v_{j},
\end{eqnarray*}
so with this choice of $X$ and $Y$ we have $Bz = BCXz + YABz$ for all $z \in \mathbb{F}^{p}$, hence $B = BCX + YAB$.
\end{proof}
This proof appears to be much simpler than the one given by Tian and Styan in [1], and it is also constructive.

\begin{example}
We present an example of the procedure given in Theorem 6. Consider three matrices
\[
A = 
\begin{bmatrix}
1 & 1 \\
1 & 1 \\
0 & 0 \\
\end{bmatrix}
\hspace{3mm}
B = 
\begin{bmatrix}
1 & 2 & 3 \\
0 & 1 & 0 \\
\end{bmatrix}
\hspace{3mm}
C = 
\begin{bmatrix}
1 & 1 \\
0 & -1 \\
1 & 0 \\
\end{bmatrix}
\]
Then
\[
AB = 
\begin{bmatrix}
1 & 3 & 3 \\
1 & 3 & 3 \\
0 & 0 & 0 \\
\end{bmatrix}
\hspace{3mm}
BC =
\begin{bmatrix}
4 & -1 \\
0 & -1 \\
\end{bmatrix}
\hspace{3mm}
ABC = 
\begin{bmatrix}
4 & -2 \\
4 & -2 \\
0 & 0 \\
\end{bmatrix}
\]
So $\rank(AB) + \rank(BC) = \rank(ABC) + \rank(B) = 3$ -- equality holds in the Frobenius rank inequality. In this case we can choose
\[
D_{B} =
\begin{bmatrix}
1 & 2 \\
0 & 1 \\
\end{bmatrix}
\hspace{3mm}
\mbox{ so that }
\hspace{3mm}
AD_{B} =
\begin{bmatrix}
1 & 3 \\
1 & 3 \\
0 & 0 \\
\end{bmatrix}
\]
so $\Rg(B)\cap\Ker(A)$ is spanned by $v_{j} = (-1, 1) \in \mathbb{R}^{2}$, which can be extended to a basis $\{(-1,1), (1,0)\}$ for $\Rg(B)$. Hence $T_{Y}$ is defined by $T_{Y}(1,1,0) = (1,0)$ and we can choose
\[
Y = \begin{bmatrix}
1 & 0 & 0 \\
0 & 0 & 0 \\
\end{bmatrix}.
\]
Now we find the $\tilde{v}_{j}$ such that $D_{B}v_{j} = BC\tilde{v}_{j}$:
\begin{eqnarray*}
\tilde{v}_{j} &=& (BC)^{-1}D_{B}v_{j} \\
&=& -\frac{1}{4}
\begin{bmatrix}
-1 & 1 \\
0 & 4 \\
\end{bmatrix}
\begin{bmatrix}
1 & 2 \\
0 & 1 \\
\end{bmatrix}
\begin{bmatrix}
-1 \\
1 \\
\end{bmatrix} \\
&=&
\begin{bmatrix}
-\frac{1}{2} \\
-1
\end{bmatrix}
\end{eqnarray*}
Then $T_{M}: \Rg(B) \to \mathbb{R}^{2}$ is defined by $T_{M}(-1,1) = (-1/2, -1)$ and $T_{M}(1,0) = (0,0)$, so the matrix of $T_{M}$ is
\[ 
M = \begin{bmatrix}
0 & -1/2 \\
0 & -1 \\
\end{bmatrix}
\]
and the proof of Theorem 6 yields
\[
X = MB = 
\begin{bmatrix}
0 & -\frac{1}{2} & 0 \\
0 & -1 & 0 \\
\end{bmatrix}.
\]
Finally,
\[ 
BCX =
\begin{bmatrix}
0 & -1 & 0 \\
0 & 1 & 0 \\
\end{bmatrix}
\hspace{3mm}
\mbox{ and }
\hspace{3mm}
YAB = 
\begin{bmatrix}
1 & 3 & 3 \\
0 & 0 & 0 \\
\end{bmatrix}
\]
so that $B = BCX + YAB$ as desired.
\end{example}

 We note here that the solution $X$ and $Y$ to (2) is not unique, although this is obvious as the bases we chose in the proof of Theorem 5 were not at all canonical. Indeed, suppose that $X$ and $Y$ are matrices that solve (2) with respect to the bases $\{P_{1},P_{2}\}$ for $X$ and $\{Q_{1},Q_{2}\}$ for $Y$. Let $\{P_{1}^{'}, P_{2}^{'}\}$ be an alternative choice of bases for $X$, and let $\{Q_{1}^{'}, Q_{2}^{'}\}$ be an alternative choice of bases for $Y$. Now in order to calculate $X$ with respect to these new bases, let us denote the matrix representation of the identity transformation $\id: \mathbb{F}^{q} \to \mathbb{F}^{q}$ with respect to $P_{1}$ and $P_{1}^{'}$ by $[I]_{P_{1}}^{P_{1}^{'}}$. Similarly, denote the matrix representation of the identity transformation $\id: \mathbb{F}^{p} \to \mathbb{F}^{p}$ with respect to $P_{2}$ and $P_{2}^{'}$ by $[I]_{P_{2}}^{P_{2}^{'}}$. Then the matrix
 \[
 [X]_{P_{1}^{'}}^{P_{2}^{'}} = ([I]_{P_{2}}^{P_{2}^{'}})^{-1} X [I]_{P_{1}}^{P_{1}^{'}},
 \]
 together with the corresponding matrix for $Y$, also solves (2) with a different choice of basis. Thus, this procedure allows one to generate infinitely many pairs of solutions to the matrix equation $B = BCX + YAB$.

\bibliographystyle{unsrt}  
\bibliography{references}

\begin{thebibliography}{1}

\bibitem{Tian:2002}
Yongge Tian and George P.~H. Styan.
\newblock When does rank(abc) = rank(ab) + rank(bc) - rank(b) hold?
\newblock {\em International Journal of Mathematical Education in Science and
  Technology}, 33(1):127--137, 2002.

\bibitem{Roth:1951}
William~E. Roth.
\newblock The equations $ax - yb = c$ and $ax - xb = c$ in matrices.
\newblock {\em Proc. Amer. Math. Soc}, (3):392--396, 1951.

\end{thebibliography}
\cite{Tian:2002, Roth:1951}

\end{document}